\documentclass{article}

\usepackage{graphicx} 
\usepackage{subfigure}

\graphicspath{{images/}}

\usepackage{natbib}

\usepackage{algorithm}
\usepackage{algorithmic}

\usepackage{hyperref}

\usepackage[accepted]{icml2012}

\usepackage{amsmath}
\usepackage{amsfonts}
\usepackage{amssymb}
\usepackage{amsthm}
\usepackage{mystyle}

\newcommand{\x}{X}
\newcommand{\z}{Z}
\newcommand{\A}{\Aa}

\renewcommand{\l}[1]{#1^{(\ell)}}
\renewcommand{\ll}[1]{#1^{(\ell+1)}}

\icmltitlerunning{Risk estimation for matrix recovery with spectral regularization}

\begin{document}

\onecolumn
\icmltitle{Risk estimation for matrix recovery with spectral regularization}

\icmlauthor{Charles-Alban Deledalle, Samuel Vaiter, Gabriel Peyr\'e}{Deledalle@ceremade.dauphine.fr}
\icmladdress{CEREMADE, CNRS, Universit\'{e} Paris-Dauphine, France}
\icmlauthor{Jalal Fadili}{}
\icmladdress{GREYC, CNRS-ENSICAEN-Universit\'{e} de Caen, France}
\icmlauthor{Charles Dossal}{}
\icmladdress{IMB, Universit\'{e} Bordeaux 1, France}

\icmlkeywords{Risk estimation, SURE, matrix recovery, matrix completion, spectral regularization, nuclear norm, proximal algorithms}

\vskip 0.3in

\begin{abstract}
In this paper, we develop an approach to recursively estimate the quadratic risk for matrix recovery problems regularized with spectral functions. Toward this end, in the spirit of the SURE theory, a key step is to compute the (weak) derivative and divergence of a solution with respect to the observations. As such a solution is not available in closed form, but rather through a proximal splitting algorithm, we propose to recursively compute the divergence from the sequence of iterates. A second challenge that we unlocked is the computation of the (weak) derivative of the proximity operator of a spectral function. To show the potential applicability of our approach, we exemplify it on a matrix completion problem  to objectively and automatically select the regularization parameter.
\end{abstract}

\section{Introduction}
\label{sec:introduction}

Consider the problem of estimating a matrix $\x_0 \in \RR^{n_1 \times n_2}$
from $P$ noisy observations
$y = \A(\x_0) + w \in \RR^{P}$,
where $w \sim \Nn(0, \sigma^2 \Id_P)$.
The linear bounded operator $\A: \RR^{n_1 \times n_2} \to \RR^P$ entails loss of information such that the problem is ill-posed.
This problem arises in various research fields.
Because of ill-posedness, side information through a regularizing term is necessary. We thus consider the problem
\begin{align}\label{eq:optim}
  \x(y) \in \uArgmin{\x \in \RR^{n_1 \times n_2}} \frac{1}{2}\|y - \A(\x)\|^2 + \lambda J(\x)
\end{align}
where the set of minimizers is assumed non-empty, $\la>0$ is a regularization parameter and $J: \RR^{n_1 \times n_2} \to \RR \cup \{\infty\}$ is a proper lower semi-continuous (lsc) convex regularizing function that imposes the desired structure on $\x(y)$.
In this paper, we focus on the case where $J$ is a convex spectral function, that is a symmetric convex function of the singular values of its argument. Spectral regularization can account for prior knowledge on the spectrum of $\x_0$, typically low-rank \citep[see e.g.][]{Fazel02}.

In practice, the choice of the regularization parameter $\lambda$ in \eqref{eq:optim} remains an important problem largely unexplored. Typically, we want to select $\la$ minimizing the quadratic risk $\EE_w\norm{\x(y) - \x_0}^2$.
Since $\x_0$ is unknown and $\x(y)$ is non-unique, one can instead consider an unbiased estimate of the \emph{prediction} risk $\EE_w \norm{\A(\x(y)) - \A(\x_0)}^2$, where it can be easily shown that $\mu(y) = \A(\x(y))$ is a single-valued mapping. With the proviso that $\mu(y)$ is weakly differentiable, the SURE \citep[for Stein unbiased risk estimator,][]{stein1981estimation}
\begin{align}\label{eq:sure}
  \SURE(y) =
  \norm{y - \mu(y))}^2
  \!-\! P \sigma^2
  \!+\! 2 \sigma^2 \diverg \mu(y)
\end{align}
is an unbiased estimate of the prediction risk,
where $\diverg \mu(y)\!=\!\Tr\pa{\!\partial \mu(y)}$, and $\partial \mu(y)$ stands for the (weak) Jacobian of $\mu(y)$.
The SURE depends solely on $y$, without prior
knowledge of $\x_0$ and then can prove very useful as a basis
for automatic ways to choose the regularization parameters $\lambda$.

\paragraph{Contributions.}
Our main contribution is to provide the derivative of matrix-valued spectral functions where the matrices have distinct singular values which extends the result of \citet{lewis-twice-spectral} to non-symmetric square matrices. This result is used to recursively compute the derivative of any solution of spectrally regularized inverse problems by solving \eqref{eq:optim}. This is achieved by computing the derivatives of the iterates provided by a proximal splitting algorithm. In particular, this provides an estimate of $\diverg \mu(y)$ in \eqref{eq:sure} which allows to compute $\SURE(y)$. A Numerical example on a matrix completion problem is given to support our findings.

\section{Recursive risk estimation}\label{sec:prox}

\paragraph{Proximal splitting}
Proximal splitting algorithms have become extremely popular to solve non-smooth convex optimization problems that arise often in inverse problems, e.g.~\eqref{eq:optim}. These algorithms provide a sequence of iterates $\l{\x}(y)$ that provably converges to a solution $\x(y)$. A practical way to compute $\diverg \mu(y)$, hence $\SURE(y)$, as initiated by \citet{vonesch-sure}, and that we pursue here, consists in differentiating this sequence of iterates.
This methodology has been extended to a wide class of proximal splitting schemes in \citep{deledalle-2012-670213}.
For the sake of clarity, and without loss of generality, we focus on the case of the forward-backward (FB) splitting algorithm \citep{CombettesWajs05}.

The FB scheme is a good candidate to solve \eqref{eq:optim} if $J$ is simple, meaning that its proximity operator has a closed-form. Recall that the proximity operator of a lsc proper convex function $G$ on $\RR^{n_1 \times n_2}$ is
\begin{align*}
  \Prox_{G}(\x) = \uargmin{\z \in \RR^{n_1 \times n_2}} \frac{1}{2}\norm{\x-\z}_F^2 + G(\z).
\end{align*}
The FB algorithm iteration reads
\begin{align}\label{eq:fb}
  \ll{\x} \!=\! \Prox_{\tau \la J}(\l{\x} + \tau \A^*(y - \A(\l{\x})))
\end{align}
where $\A^*$ denotes the adjoint operator of $\A$, $\tau>0$ is chosen such that $\tau \norm{\A^* \A} < 2$, the dependency of the iterate $\l{\x}$ to $y$ is dropped to lighten the notation.

\paragraph{Risk estimation}
The divergence term $\diverg \mu(y)$ is obtained by differentiating formula \eqref{eq:fb}, which allows, for any vector $\delta \in \RR^P$ to compute iteratively $\l{\xi} = \partial \l{\x}(y)[\de]$ (the derivative of $y \mapsto \l{\x}(y)$ at $y$ in the direction $\de$) as
\begin{align*}
  & \ll{\xi} =
  \partial \Prox_{\tau \la J}(\l{\Xi})[\l{\zeta}]\\
  \text{where} \quad
  &
  \l{\Xi} = \l{\x} + \tau \A^*(y - \A(\l{\x}))\\
  \text{and} \quad
  &
  \l{\zeta} =
  \l{\xi} + \tau \A^*(\delta - \A(\l{\xi})).
\end{align*}
Using the Jacobian trace formula of the divergence,
it can be easily seen that
\begin{align}
  \label{eq-randomized-df}
  \small
  \diverg \mu(y) = \EE_\de \dotp{\partial \mu(y)[\de]}{\de} \approx \frac{1}{k} \sum_{i = 1}^k \dotp{\partial \mu(y) [\de_i]}{\de_i}
\end{align}%
where $\de \sim \Nn(0,\Id_P)$ and $\de_i$ are $k$ realizations of $\de$.
The $\SURE(y)$ can in turn be iteratively estimated by plugging
$\partial \mu(y)[\de_i] = \A(\partial \l{\x}(y)[\de_i])$ in \eqref{eq-randomized-df}.

\section{Local behavior of spectral functions}
\label{sec:sf}

This section studies the local behavior of real- and matrix-valued spectral functions. We write the (full) singular value decomposition (SVD) of a matrix $\x \in \RR^{n_1 \times n_2}$
\eq{
	\x = V_\x \diag(\La_\x) U_\x^*
}
(which might not be in general unique), where $\La_\x \in \RR^{n}$
is the vector of singular values of $\x$ with $n = \min (n_1, n_2)$,
$\diag(\La_\x) \in \RR^{n_1 \times n_2}$
denotes the rectangular matrix with entries $\La_\x$ on its main diagonal and 0 otherwise, 
and $V_\x \in \RR^{n_1 \times n_1}$ and $U_\x \in \RR^{n_2 \times n_2}$ are the unitary matrices of left and right singular vectors.

\subsection{Scalar-valued Spectral Functions}

A real-valued spectral function $J$ can by definition be written as
\begin{align}\label{eq:rsf}
  J(\x) = \phi(\La_\x)
\end{align}
where $\phi : \RR^n \rightarrow \RR$ is a symmetric function of its argument,
meaning $\phi(P \La)=\phi(\La)$ for any permutation matrix $P \in \RR^{n \times n}$ and $\Lambda$ in the domain of $\phi$. We extend $\phi$ to the negative half-line as $\phi(\La)=\phi(|\La|)$.

We then consider $J$ a scalar-valued spectral function as defined in \eqref{eq:rsf}. From subdifferential calculus on spectral functions \citet{lewis-convex-unitarily}, we get the following.
\begin{prop}\label{prop-subdiff-spectral}
  A spectral function $J(\x)=\phi(\La_\x)$ is convex
  if and only if $\phi$ is convex, and then
  \eq{
    \forall \ga>0, \quad \Prox_{\ga J}(\x) = V_\x \diag( \Prox_{\ga \phi}(\La_\x) ) U_\x^*.
  }
\end{prop}%

\subsection{Matrix-valued Spectral Functions}

We now turn to matrix-valued spectral functions
\eql{\label{eq-matrix-valued}
	F(\x) = V_\x \diag( \Phi(\La_\x) ) U_\x^*,
}
where $\Phi : \RR^n \rightarrow \RR^n$ is symmetric in its arguments, meaning $\Phi \circ P = P \circ \Phi$ for any permutation matrix $P \in \RR^{n \times n}$.
We extend $\Phi$ to negative numbers as
$\Phi(\La)=\sign(\La) \odot \Phi(|\La|)$ and
$\odot$ is the entry-wise matrix multiplication.
One can observe that for $F(\x) = \Prox_{\ga J}(\x)$ with
$\Phi = \Prox_{\ga \phi}$, the proximity operator of a convex
scalar-valued spectral function is a matrix-valued spectral function.

The following theorem provides a closed-form expression of the derivative of $F$ when $X$ is square, i.e. $n_1=n_2=n$, with distinct singular values.

\begin{thm}\label{prop:diff_msf}
  For any matrix-valued spectral function $F$ in \eqref{eq-matrix-valued}, let the quantity
  \begin{align*}
    & \foralls \de \in \RR^{n_1 \times n_2}, \quad
    D(\x)[\de] =
    V_\x
    (\Mm (\La_\x) [ \bar \de ]
    + \Ga_S(\La_\x) \odot \Pp_S(\bar \de)
    + \Ga_A(\La_\x) \odot \Pp_A(\bar \de))
    U_\x^*
  \end{align*}
  where $\bar \de = V_X^* \de U_X \in \RR^{n_1 \times n_2}$,
  the symmetric and anti-symmetric parts are defined,
  for $1 \le i \le n_1$ and $1 \le j \le n_2$, as
  \begin{align*}
    &\Pp_S(Y)_{i,j} =
    \choice{
      \frac{Y_{i,j}+Y_{j,i}}{2} &
      \text{if} \quad i \le n \qandq j \le n\\
      \frac{Y_{i,j}}{2} &
      \text{otherwise}
    }\\
    \qandq
    &
    \Pp_A(Y)_{i,j} =
    \choice{
      \frac{Y_{i,j}-Y_{j,i}}{2} &
      \text{if} \quad i \le n \qandq j \le n\\
      \frac{Y_{i,j}}{2} &
      \text{otherwise}
    }
  \end{align*}
  and $\Mm(\La) : \RR^{n_1 \times n_2} \mapsto \RR^{n_1 \times n_2}$ is
  \eq{
    \Mm(\La) = \diag \circ \partial \Phi(\La) \circ \diag.
  }
  The matrices $\Ga_S(\La)$ and $\Ga_A(\La)$ are defined,
  for all $1 \leq i \leq n_1$ and $1 \leq j \leq n_2$, as
  \begin{align*}
    \Ga_S(\La)_{i,j} &=
    \choice{
      0 & \text{if} \quad i=j\\
      \frac{\Phi(\La)_i-\Phi(\La)_j}{\La_i-\La_j}
      & \text{if} \quad \La_i \neq \La_j \\
      \partial \Phi(\La)_{i,i} - \partial \Phi(\La)_{i,j}
      & \text{otherwise},
    } \\
    \Ga_A(\La)_{i,j} &=
    \choice{
      0 & \text{if} \quad i=j\\
      \frac{\Phi(\La)_i+\Phi(\La)_j}{\La_i+\La_j}
      & \text{if} \quad \La_i > 0 \quad \text{or} \quad \La_j > 0\\
      \partial \Phi(\La)_{i,i} - \partial \Phi(\La)_{i,j}
      & \text{otherwise}.
    }
  \end{align*}
  where for $i > n$ we have extended $\La$ and $\Phi(\La)$ as
  $\La_i = 0$ and $\Phi(\La)_i = 0$.\\
  Assume that $\x$ is a square matrix, i.e. $n_1 = n_2 = n$,
  and with distinct singular values, such that $\La_i \ne \La_j$
  for all $i \ne j$.
  Then, a matrix-valued spectral function $F$ is differentiable at $\x$
  if and only if $\Phi$ is differentiable at $\La_\x$.
  Moreover,
 \begin{align*}
    & \foralls \de \in \RR^{n \times n}, \quad
    \partial F(\x)[\de] = D(\x)[\de]
 \end{align*}
\end{thm}
The proof is given in Appendix \ref{sec:proof}.

Theorem~\ref{prop:diff_msf} generalizes the result of \citet{lewis-twice-spectral} to square matrices that are not necessarily symmetric, and we recover their formula when $\x$ and $\de$ are symmetric matrices and $\x$ has distinct singular values. Regularity properties and expression of the directional derivative of symmetric matrix-valued separable spectral functions (possibly non-smooth) over non-necessarily symmetric matrices were also derived in \citet{Sun02}. Before revising the previous version of this manuscript, \citet{Candes12} brought to our attention their recent work on the SURE framework for parameter selection in denoising low-rank matrix data. Towards this goal, they provided closed-form expressions for the directional derivative and divergence of matrix-valued spectral functions over rectangular matrices with distinct singular values. They also addressed the case of complex-valued matrices.

Although our proof of Theorem~\ref{prop:diff_msf} is rigorously valid only for square matrices with distinct singular values, we conjecture that the formula of the directional derivative holds for rectangular matrices with repeated singular values. For the symmetric case with repeated eigenvalues, this assertion was formally proved
in \citet{lewis-twice-spectral}. As stated above, the full-rank rectangular case with distinct singular values was proved in \citet{Candes12}, where it was also shown that the divergence formula has a continuous extension to all matrices.

\section{Numerical applications}\label{sec:nnr}

\subsection{Nuclear norm regularization}\label{sec:numeric}

We here consider the problem of recovering a low-rank matrix $\x_0 \in \RR^{n_1 \times n_2}$.
To this end, $J$ is taken as the nuclear norm (a.k.a., trace or Schatten $1$-norm) which is in some sense the tightest convex relaxation to the NP-hard rank minimization problem \cite{CandesCompletion08}. The nuclear norm is defined by
\begin{align}\label{eq:nuclear_norm}
  J(\x) = \norm{\x}_* \triangleq \norm{\La_\x}_1 .
\end{align}
Taking $J(\cdot)$ as $\norm{\cdot}_*$ and $\phi$ as $\norm{.}_1$ in Proposition~\ref{prop-subdiff-spectral} gives:
\begin{cor}\label{cor:diff_msf}
The proximal operator of $\ga\norm{\cdot}_*$ is
\begin{align}\label{eq:prox_nuclear}
  \forall \ga>0, \quad
  \Prox_{\ga \norm{\cdot}_*}(\x) = V_\x \diag( T_\ga(\La_\x) ) U_\x^* ~,
\end{align}
where $T_{\ga} = \Prox_{\ga \norm{.}_1}$ is the component-wise soft-thresholding, defined for $i=1,\ldots,n$ as
\begin{align*}
  T_{\ga}(t)_i = \max(0,1-\ga/\norm{t_i})t_i.
\end{align*}
\end{cor}
We now turn to the derivative of $F = \Prox_{\ga \norm{\cdot}_*}$. A straightforward attempt is to take $\Phi = \Prox_{\ga \norm{.}_1} = T_{\ga}$ and apply Theorem~\ref{prop:diff_msf} with
\begin{align}\label{eq:diff_soft_thresholding}
  \partial \Phi(\x)[\de]_i =
  \partial T_\ga(t)[\de]_i =
  \choice{
    0 \qifq \norm{t_i} \leq \ga \\
    \de_{i} \quad\text{otherwise}.
  }
\end{align}
However, strictly speaking, Theorem~\ref{prop:diff_msf} does not apply since a proximity mapping is 1-Lipschitz in general, hence not necessarily differentiable everywhere. Thus, its derivative may be set-valued, as is the case for soft-thresholding at $\pm \ga$.

A direct consequence of Corollary \ref{cor:diff_msf} is that $J$ is a simple function allowing for the use of the FB algorithm. Moreover, the expression of the derivative \eqref{eq:diff_soft_thresholding} provides an estimation of the SURE as explained in Section~\ref{sec:prox}.

\subsection{Application to matrix completion}\label{sec:numeric}

We now exemplify the proposed SURE computation approach on a matrix completion problem encountered in recommendation systems such as the popular Netflix problem.
We therefore consider the forward model $y = \A(\x_0) + w \in \RR^P$, $w \sim \Nn(0, \sigma^2 \Id_P)$,
where $\x_0$ is a dense but low-rank (or approximately so) matrix and $A$ is binary masking operator.

We have taken $(n_1,n_2)=(1000,100)$ and $P=25000$ observed entries (i.e., $25\%$).
The underlying dense matrix $\x_0$ has been chosen to be approximately low-rank with a rapidly decaying spectrum
$\La_{\x_0} = \{ k^{-1} \}_{k = 1}^n$.
The standard deviation $\sigma$ has been set such that
the resulting minimum least-square estimate has a relative error $\norm{\x_{LS}-\x_0}_F/\norm{\x_0}_F=0.9$.
Figure~\ref{fig:sure} depicts the prediction risk and
its $\SURE$ estimate as a function of $\la$.
For each value of $\la$ in the tested range,
$\SURE(y)$ in \eqref{eq:sure} has been computed for a single realization of $y$
with $k = 4$ realizations $\de_i$ in \eqref{eq-randomized-df}%
\footnote{Without impacting the optimal choice of $\la$, the two curves have been vertically shifted for visualization.\label{footnote:shift}}.
At the optimal $\lambda$ value, $\x(y)$ has a rank of $55$
with a relative error of $0.46$ (i.e., a gain of about a factor $2$ w.r.t.~the least-square estimator).

\begin{figure}[t]
  \centering
  \includegraphics[width=0.5\linewidth]{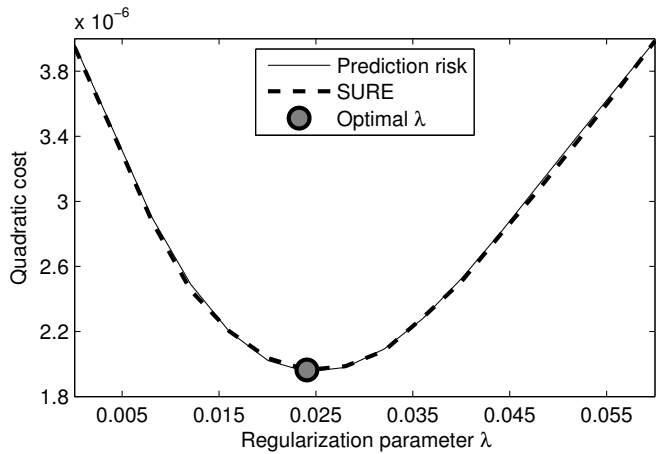}
  \caption{Predicted risk and its $\SURE$ estimate$^{\ref{footnote:shift}}$.}
  \label{fig:sure}
\end{figure}

\section{Conclusion}
\label{sec:conclusion}

The core theoretical contribution of this paper is the derivative of square matrix-valued spectral functions. This was a key step to compute the derivative of the proximal operator associated to the nuclear norm, and finally to use the $\SURE$ to recursively estimate the quadratic prediction risk of matrix recovery problems involving the nuclear norm regularization. The $\SURE$ was also used to automatically select the optimal regularization parameter.

\appendix

\section{Summary of the proof of Theorem~\ref{prop:diff_msf}}\label{sec:proof}

The following lemma derives the expression of the derivative of the SVD mapping $\x \mapsto (V_\x,\La_\x,U_\x)$. Note that this mapping is not well defined because even if the $\La_\x$ are distinct, one can apply arbitrary sign changes and permutations to the set of singular vectors. The lemma should thus be interpreted in the sense that one can locally write a Taylor expansion using the given differential for any particular choice of SVD.
We point out that a proof of this lemma using wedge products can be found in \citep{edelman-handout}, but for the sake of completeness, we provide our own proof here.

\begin{lem}\label{prop-svd-derivative}
	We consider $X_0 \in \RR^{n \times n}$ with distinct singular values.
        For any matrix $X$ in a neighborhood of $X_0$,
        we can define without ambiguity the SVD mapping $X \mapsto (V_X,\La_X,U_X)$
        by sorting the values in $\La_X$ and imposing sign constraints on $U_X$.
        The singular value mapping $X \mapsto (V_X,\La_X,U_X)$ is $C^1$ and
        for a given matrix $\de$, its directional derivative is
	\begin{align*}\label{eq-derivative-singulval}
		\partial \La_X[\de] &= \diag( V_X^* \de U_X )~,\\
		\partial V_X[\de] &= V_X \iota_V ,\\
	        \partial U_X[\de] &= U_X \iota_U
	\end{align*}
        where
        $\iota_V \in \RR^{n \times n}$ and
        $\iota_U \in \RR^{n \times n}$ are defined,
        for all $1 \leq i \leq n$ and $1 \leq j \leq n$, as
	\eql{\label{eq-derivative-singulvect}
		(\iota_V)_{i,j} =
		\frac{{(\La_X)}_j \bar \de_{i,j} + {(\La_X)}_i \bar \de_{j,i}}{{(\La_X)}_j^2-{(\La_X)}_i^2}
		\qandq
		(\iota_U)_{i,j} =
		\frac{{(\La_X)}_i \bar \de_{i,j} + {(\La_X)}_j \bar \de_{j,i}}{{(\La_X)}_j^2-{(\La_X)}_i^2}~.
	}
        and where $\bar \de = V_X^* \de U_X \in \RR^{n \times n}$.
\end{lem}

\begin{proof}
  Let $S_n$ be the sub-space of
  Hermitian matrix in $\RR^{n \times n}$.
  Let $\psi : \RR^{n \times n} \times \Yy \rightarrow \RR^{n \times n} \times S_{n} \times S_{n}$, where
  $\Yy = \RR^{n \times n} \times \RR^{n \times n} \times \RR^{n}$, be defined for
  $Y = (V, U, \La) \in \Yy$ as
  \begin{align*}
    \psi(X, Y) =
    \left(\begin{array}{ccc}
      X - V \diag(\La) U^*, &
      V^* V - \Id, &
      U^* U - \Id
    \end{array}\right).
  \end{align*}
  We have for any vector $\ze_X \in \RR^{n \times n}$
  \begin{align}
    \label{eq-proof-svd-diff1}
    \partial_1 \psi(X, Y) [\ze_X]
    =
    \left(\begin{array}{ccc}
      \ze_X, &
      0, &
      0
    \end{array}\right)
  \end{align}
  and for any vector $\ze_Y = (\ze_U, \ze_V, \ze_\La) \in \Yy$
  \begin{align*}
    \partial_2 \psi(X, Y) [\ze_Y]
    =
    \left(\begin{array}{ccc}
      -\ze_V \diag(\La) U^* - V \diag(\La) \ze_U^* - V \diag(\ze_\La) U^*, &
      V^* \ze_V + \ze_V^* V, &
      U^* \ze_U + \ze_U^* U
    \end{array}\right).
  \end{align*}
  Let $X_0$ have distinct singular values
  and any of its SVD $Y_0 = (U_0, V_0, \La_0)$.
  We have $\psi(X_0, Y_0) = 0$.
  Moreover, denoting $\iota_V = V_0^* \ze_V \in \RR^{n \times n}$ and $\iota_U = U_0^* \ze_U \in \RR^{n \times n}$,
  for any $z = (z_1, z_2, z_3) \in \RR^{n \times n} \times S_{n} \times S_{n}$, solving $\partial_2 \psi(X_0, Y_0) [\ze_Y] = z$ is equivalent to solving
  \begin{align}
    \iota_V \diag(\La_0) + \diag(\La_0) \iota_U^* + \diag(\ze_\La) &= -V_0^* z_1 U_0 = -\bar{z}_1 \label{eq-proof-svdbis}
  \end{align}
  where $\iota_V + \iota_V^* = z_2$ and
  $\iota_U + \iota_U^* = z_3$.
  Considering $z_2 = 0$ and $z_3 = 0$ shows that $\iota_V$ and $\iota_U$ are antisymmetric.
  In particular they are zero along the diagonal. Thus applying the operator $\diag$ to both sides of \eqref{eq-proof-svdbis} shows
  $\ze_\La = -\diag(V_0^* z_1 U_0)$.
  Now considering the entries $(i,j)$ and $(j,i)$ of the linear system \eqref{eq-proof-svdbis} shows that for any $1 \leq i \leq n$ and $1 \leq j \leq n$
  \begin{align}
    \label{eq-proof-svd-2b2linear-systems}
    \begin{pmatrix}
      {(\La_0)}_j & -{(\La_0)}_i\\
      -{(\La_0)}_i & {(\La_0)}_j
    \end{pmatrix}
    \begin{pmatrix}
      (\iota_V)_{i,j}\\
      (\iota_U)_{i,j}
    \end{pmatrix}
    =
    \begin{pmatrix}
      -({{\bar z}_1})_{i,j}\\
      -({{\bar z}_1})_{j,i}
    \end{pmatrix}.
  \end{align}
  Since for $i \ne j$, ${(\La_0)}_i \ne {(\La_0)}_j$,
  these $2 \times 2$ symmetric linear systems can be solved.
  Then $\partial_2 \psi(X_0, Y_0)$ is invertible
  on $\RR^{n \times n} \times 0_n \times 0_n$
  and for $z = (z_1, 0, 0)$, its inverse is
  \begin{align}
    \label{eq-proof-svd-idiff2}
    \left(\partial_2 \psi(X_0, Y_0)\right)^{-1}[z]
    =
    \left(
    \begin{array}{ccc}
      V_0 \iota_V, &
      U_0 \iota_U, &
      -\diag(V_0^* z_1 U_0)
    \end{array}
    \right)
  \end{align}
  where $\iota_V$ and $\iota_U$ are given by the solutions of
  the above series of $2 \times 2$ symmetric linear systems.

  Since $\Im(\partial_1 \psi(X,Y(X))) \subset \RR^{n \times n} \times 0_n \times 0_n$,
  we can apply the implicit function theorem \cite{RockafellarWets}.
  Hence, for any $X \in \RR^{n \times n}$
  in the neighborhood of $X_0$, there exists a function
  $Y(X) = (U_X, V_X, \La_X)$ such that $\psi(X, Y(X)) = 0$, i.e.~%
  $X$ admits an SVD. Moreover, this function is $C^1$ in the
  neighborhood of $X_0$ and
  its differential is
  \begin{align*}
    \partial Y(X) = - \partial_2 \psi(X,Y(X))^{-1} \circ \partial_1 \psi(X,Y(X)).
  \end{align*}
  Injecting \eqref{eq-proof-svd-diff1} and \eqref{eq-proof-svd-idiff2}
  gives the desired formula by solving \eqref{eq-proof-svd-2b2linear-systems} in closed form.
  Since $X_0$ is any matrix with distinct
  singular values, we can conclude.
\end{proof}

We now turn to the proof of the theorem.

\begin{proof}
  Since the singular values of $\x$ are all distinct,
  by composition of differentiable functions, we can
  derive the relationship \eqref{eq-matrix-valued} that defines $F$
  which gives
  \eq{
    V_X^* \partial F(X)[\de] U_X =
    \iota_V \diag(\Phi(\La_X))
    +
    \diag(\Phi(\La_X)) \iota_U^*
    + \Mm(\La_X)[\bar \de]
  }
  where we have used the notation introduced in Lemma \ref{prop-svd-derivative}.
  Using the expression \eqref{eq-derivative-singulvect} for $\iota_U$ and $\iota_V$ shows that
  the matrix $W=\iota_V \diag(\Phi(\La_X)) + \diag(\Phi(\La_X)) \iota_U^*$ is computed as
  \eq{
    W_{i,j} =
    \frac{1}{\La_j^2-\La_i^2}\pa{
      \phi_j ( \La_j \bar \de_{i,j} + \La_i \bar\de_{j,i} )
      -
      \phi_i ( \La_i \bar \de_{i,j} + \La_j \bar\de_{j,i} )
    }
  }
  where $\phi=\Phi(\La)$.
  Rearranging this expression using the symmetric and anti-symmetric parts shows the desired formula.
\end{proof}

\small
\bibliography{bibliography}
\bibliographystyle{icml2012}

\end{document}